\documentclass{amsart}
\usepackage{amsmath}
\usepackage{expl3}
\usepackage{amssymb}
\usepackage{amsthm}
\usepackage[british]{babel}
\usepackage{amscd,amsbsy}
\usepackage{mathtools}
\usepackage{xcolor}
\usepackage{tikz-cd}
\usetikzlibrary{arrows,matrix}
\usepackage{float}
\usepackage{mathtools}
\usepackage{enumitem}
\usepackage{mathrsfs}  
\usepackage{stmaryrd}
\usepackage[normalem]{ulem}
\usepackage[all]{xy}

\usepackage{hyperref}
\usepackage{color}
\definecolor{darkred}{rgb}{.5,0,0}
\definecolor{darkgreen}{rgb}{0,.5,0}
\definecolor{darkblue}{rgb}{0,0,.5}
\hypersetup{
    pdfencoding=unicode,
    colorlinks=true,
    linkcolor=darkred,
    citecolor=darkgreen,
    urlcolor=darkblue,
    bookmarks=false,
    pdftoolbar=true,
    pdfmenubar=true,
    pdffitwindow=true,
    pdfstartview={FitH},
    pdfcreator={},
    pdfproducer={},
    pdfnewwindow=true,
    pdftitle={Cm solutions of semialgebraic equations on curves},
    pdfauthor={Edward Bierstone, Jean-Baptiste Campesato},
}

\DeclareSymbolFont{calletters}{OMS}{cmsy}{m}{n}
\DeclareSymbolFontAlphabet{\mathcal}{calletters}

\newtheorem{theorem}{Theorem}[section]
\newtheorem{lemma}[theorem]{Lemma}
\newtheorem{proposition}[theorem]{Proposition}
\newtheorem{corollary}[theorem]{Corollary}
\newtheorem{claim}[theorem]{Claim}

\theoremstyle{definition}

\newtheorem{definitions}[theorem]{Definitions}
\newtheorem{example}[theorem]{Example}

\newtheorem{definitions and remarks}[theorem]{Definitions and Remarks}

\theoremstyle{remark}

\newtheorem{remark}[theorem]{Remark}
\newtheorem{remarks}[theorem]{Remarks}

\numberwithin{equation}{section}

\renewcommand\o[1]{\underset{\overset{#1}{}}{o}}
\renewcommand\O[1]{\underset{\overset{#1}{}}{O}}

\newcommand{\bdry}{\mathrm{bdry}\,}

\newcommand{\Ker}{\mathrm{Ker}\,}

\newcommand{\al}{{\alpha}}
\newcommand{\be}{{\beta}}
\newcommand{\de}{{\delta}}
\newcommand{\ep}{{\varepsilon}}

\newcommand{\ga}{{\gamma}}
\newcommand{\Ga}{{\Gamma}}

\newcommand{\La}{{\Lambda}}

\newcommand{\p}{{\partial}}
\newcommand{\s}{{\sigma}}

\newcommand{\vp}{{\varphi}}

\newcommand{\IN}{{\mathbb N}}

\newcommand{\IR}{{\mathbb R}}

\newcommand{\IZ}{{\mathbb Z}}

\newcommand{\cC}{{\mathcal C}}

\newcommand{\cH}{{\mathcal H}}

\newcommand{\cK}{{\mathcal K}}

\newcommand{\cP}{{\mathcal P}}

\newcommand{\ta}{{\widetilde a}}
\newcommand{\tb}{{\widetilde b}}

\newcommand{\tf}{{\widetilde f}}
\newcommand{\tg}{{\widetilde g}}

\newcommand{\tA}{{\widetilde A}}

\newcommand{\tB}{{\widetilde B}}

\newcommand{\tG}{{\widetilde G}}

\newcommand{\RN}[1]{%
  \textup{\uppercase\expandafter{\romannumeral#1}}%
}

\begin{document}
\title[$\cC^m$ solutions of semialgebraic equations on curves]
{$\cC^m$ solutions of semialgebraic equations on curves}
\author[E.~Bierstone]{Edward Bierstone}
\address[E.~Bierstone]{University of Toronto, Department of Mathematics, 40 St. George Street, Toronto, ON, Canada M5S 2E4}
\email{bierston@math.utoronto.ca}
\author[J.-B.~Campesato]{Jean-Baptiste Campesato}
\address[J.-B.~Campesato]{Universit\'e Angers, CNRS, LAREMA, SFR MATHSTIC, F-49000 Angers, France}
\email{jb.campesato@univ-angers.fr}
\thanks{Research supported by NSERC Discovery Grant RGPIN-2017-06537 (Bierstone) and 
ANR project NewMIRAGE ANR-23-CE40-0002 (Campesato).}
\date{5th October 2026}


\begin{abstract}
Consider a system of equations $A(x)\cdot F(x) = B(x)$ on a subset $X$ of $\IR^n$, where
$A(x)$ and $B(x)$ are matrix- and vector-valued semialgebraic functions on $X$, and the unknown $F(x)$ is
a field of vector-valued $m$-jets. 
We assume there is a solution which is the field of Taylor polynomials of order $m$ on $X$
of a $\cC^m$ vector-valued function $f$ on $\IR^n$, and ask whether we can find a $\cC^m$ semialgebraic solution $f$.
Our main result is a positive answer in the case $\dim X = 1$. The methods are based on an article of 
Fefferman and Luli for $X \subset \IR^2$, and a secondary goal is to show that their approach applies in a much simpler way in the case
$\dim X =1$, or in the case of the semialgebraic Whitney extension problem in $\IR^2$. Our results hold more generally
for functions definable in an o-minimal expansion of $\IR$.
\end{abstract}

\maketitle

\setcounter{tocdepth}{1}
\tableofcontents

\section{Introduction}\label{sec:intro}
\subsection{Overview}\label{subsec:overview}
This paper concerns the solutions of a system of semialgebraic linear equations $A(x)\cdot F(x) = B(x)$ on a semialgebraic subset $X$ of $\IR^n$.
Here $A(x)$ and $B(x)$ are matrix- and vector-valued semialgebraic functions on $X$ (respectively), while the unknown $F(x)$ is
a vector-valued function on $X$, or, more generally, a field of vector-valued
$m$-jets on $X$ (an $m$-\emph{jet} is a Taylor polynomial of order $m$).

We assume that the system of equations has a solution which is the field of Taylor polynomials of order $m$ on $X$
of a $\cC^m$ vector-valued function $f$ on $\IR^n$, and ask whether we can find such a solution (in particular, with the same
order of differentiability $m$), where, in addition, $f$ is semialgebraic.  
An important special case of the problem is the $\cC^m$
semialgebraic Whitney extension problem. Precise formulations of these problems are given in \S\ref{subsec:main}.

Our main new result is a positive answer to the preceding question in the case $\dim X = 1$. See Theorems
\ref{thm:mainjet} and \ref{thm:main}, as well as
Corollary \ref{cor:mainB}, for our most general assertions. Positive answers
were previously given by Fefferman and Luli in the case that $X$ is a semialgebraic subset of the plane $\IR^2$
\cite{FL} and, by the authors together with P. Milman \cite{BCM}, for general semialgebraic $X\subset \IR^n$ provided 
that one allows a certain controlled loss of differentiability,

The semialgebraic Whitney extension problem is to find conditions under which a
semialgebraic function $g:X\to\IR$ has a semialgebraic $\cC^m$ extension to $\IR^n$. A solution for one-dimensional 
$X\subset \IR^n$ or for $X\subset \IR^2$ follows from Theorem \ref{thm:mainjet}. The $\cC^m$ semialgebraic
Whitney extension problem was originally proposed by Bierstone and Milman in \cite{Z}, and the more general problem concerning
a system of equations by Fefferman and Koll\'ar \cite{FK}. Wies{\l}aw Paw{\l}ucki asked us about the one-dimensional case at a
conference in his honour in June, 2025.

Our methods are closely based on those of \cite{FL}, and a secondary goal of the article, which was part of our original
motivation, is to show that the approach of Fefferman and Luli applies in a much simpler way in the case that $\dim X =1$,
or in the case of the semialgebraic Whitney extension problem in $\IR^2$. In particular, we do not need Fefferman's main
$\cC^m$ extension theorem \cite{F} or the ideas of a $\cC^m$ bundle and Glaeser refinement to prove our main theorems. 
We will, nevertheless, show that the problem on solutions of equations treated in Theorem \ref{thm:main} is equivalent to the formulation
of \cite{FL} in terms of sections of bundles; see Corollary \ref{cor:mainC}. We hope that our paper
may help a reader better understand the methods of \cite{FL}, and lead to progress on the question for general $X\subset\IR^n$.

It is convenient to formulate our main problem in terms of $\cC^m$ Whitney fields on $X$, rather than in terms of $\cC^m$
functions on $\IR^n$ (see \S\S\ref{subsec:whitney},\,\ref{subsec:main}). The formulations are equivalent 
if $X$ is closed, because of Whitney's
classical extension theorem \cite{W} and its semialgebraic (or definable) counterpart due to Kurdyka and Paw{\l}ucki \cite{KP}.

For simplicity of language, we work with semialgebraic sets and functions, but all of our results hold more generally for sets
and functions that are definable in an o-minimal expansion of $\IR$ (see Section \ref{sec:omin}).

\subsection{Notation}\label{subsec:notn}
The notation that we use is chosen to be close to that of \cite{FL}, although there are some differences. $\cC^m(\IR^n,\IR^D)$
denotes the space of $\IR^D$-valued functions on $\IR^n$ with continuous derivatives up to order $m$. We write
$\cC^m(\IR^n) := \cC^m(\IR^n,\IR)$.

Let $\cP$ denote the vector space of real-valued polynomials $\sum_{|\al|\leq m} c_\al T^\al/\al!$
of degree $\leq m$ in $n$ variables $T= (T_1,\ldots,T_n)$. Here $\al$ is a multiindex $\al = (\al_1,\ldots,\al_n) \in \IN^n$,
$T^\al := T_1^{\al_1}\cdots T_n^{\al_n}$, $\al ! := \al_1!\cdots\al_n!$
and $|\al| := \al_1 +\cdots +\al_n$. The vector space $\cP$ has the structure of a ring, 
with the product $P\odot Q$ of $P,Q \in \cP$ given by truncating terms of degree $> m$
in $PQ$ as a product of polynomials. The product $P\odot Q$ is commutative and associative. Note that $\cP^D$ is a real vector space of dimension 
$d = D\binom{m+n}{m}$, and $\cP^D$ has the structure of a $\cP$-module with the componentwise product.

An $m$-\emph{jet} at $a\in \IR^n$ denotes an element $P\in \cP$, where $P$ can be regarded as a Taylor polynomial $\sum c_\al (x-a)^\al/\al!$ 
of order $m$ centred at $a\in \IR^n$, via the substitution $T = x-a$. A vector-valued $m$-jet at $a$ means an element of $\cP^D$,
for some $D$. An $m$-jet $\sum_{|\al|\leq m} c_\al T^\al/\al! \in \cP^D$ will be identified 
with the vector of coefficients $(c_\al)_{|\al|\leq m}\in \IR^d$ (i.e., coefficients of the $T^\al/\al!$).

Given $f \in \cC^m(\IR^n,\IR^D)$ and $a\in \IR^n$, the (vector-valued) $m$-jet $J_a(f)$ denotes
$\sum_{|\al|\leq m} \p^\al f(a) T^\al/\al! \in \cP^D$, where
$\p^\al := \p^{|\al|}/\p x^\al$. If $f\in \cC^m(\IR^n)$ and $g\in \cC^m(\IR^n,\IR^D)$,
then $J_a(fg) = J_a(f)\odot J_a(g)$. All products $\odot$ appearing in the paper involve two jets at the same base point $a$.

\subsection{Recall on Whitney's extension theorem}\label{subsec:whitney}
Consider a field of $m$-jets $F = (F_x)_{x\in X}$ on a subset $X\subset \IR^n$; say, $F_x = \sum_{|\al|\leq m}f_\al(x) T^\al/\al!$.

First suppose $X$ is an open subset of $\IR^n$. Then $F$ is the Taylor field (field of Taylor polynomials) of a function $f\in \cC^m(X)$
provided that, for all $a\in X$ and all $\al$ such that $|\al|\leq m$,
\begin{equation}\label{eq:Taylor}
f_\al(x) - \sum_{|\be|\leq m-|\al|} f_{\al + \be}(a)\frac{(x-a)^\be}{\be!} = o(|x-a|^{m-|\al|})\quad \text{as } x \to a \text{ in } X
\end{equation}
(converse of Taylor's theorem). For a closed set $X\subset \IR^n$, it is not, in general, true that a field of $m$-jets $F$ on $X$
is the Taylor field on $X$ of a function $f\in \cC^m(U)$ on a neighbourhood $U$ of $X$ provided that \eqref{eq:Taylor} holds
for all $a\in X$ and $|\al|\leq m$ (counterexample of Whitney in \cite[\S4]{W}).

Whitney's classical extension theorem \cite{W} asserts that, if $X$ is closed, then $F$ is the Taylor field on $X$ of a $\cC^m$
function on $\IR^n$ provided that, for all $a\in X$ and all $|\al|\leq m$,
\begin{equation}\label{eq:Whitney}
f_\al(x) - \sum_{|\be|\leq m-|\al|} f_{\al + \be}(y)\frac{(x-y)^\be}{\be!} = o(|x-y|^{m-|\al|})\quad \text{as } x, y \to a \text{ in } X.
\end{equation}

A field of $m$-jets $F$ satisfying \eqref{eq:Whitney} is called a \emph{Whitney field} of \emph{order} $m$ on $X$.
(We do not require that $X$ be closed in this definition, but Whitney's extension theorem holds only when $X$ is closed). Note that all coefficients $f_\alpha(x)$ of a Whitney field $F$ are continuous on $X$.

We say that a field of $m$-jets (or vector-valued $m$-jets) $F = (F_x)_{x\in X}$ on a subset $X\subset \IR^n$ is \emph{semialgebraic}
if all coefficients $f_\al(x)$ are semialgebraic functions. (In this case, it follows that $X$ is semialgebraic.)

\subsection{Semialgebraic equations and main results}\label{subsec:main}
Let $X$ denote a subset of $\IR^n$. We first consider a system of equations
\begin{equation}\label{eq:eqjet}
A_x\odot F_x = B_x, \quad x\in X,
\end{equation}
where $B: X \to \cP^k$ is a given semialgebraic $\IR^k$-valued field of $m$-jets on $X$, $A$ is a given $k\times D$ matrix
of semialgebraic $m$-jet fields on $X$, and the unknown $F$ is a field of $\IR^D$-valued $m$-jets on $X$. (Thus $X$ is semialgebraic,
by the Tarski-Seidenberg theorem.)

In the notation of \S\S\ref{subsec:notn},\,\ref{subsec:whitney}, $F = (F_x)_{x\in X}$, where $F_x = (F_{1,x},\ldots,F_{D,x})$ and $F_{j,x} 
= \sum_{|\al|\leq m} f_{j,\al}(x) T^\al/\al! \in \cP$, $j=1,\ldots,D$. The matrix $A = (A_{ij})$ has entries $A_{ij} = (A_{ij,x})_{x\in X}$, $i=1,\ldots,k$, $j=1,\ldots,D$,
which are semialgebraic $m$-jet fields on $X$; i.e., $A_{ij,x} = \sum_{|\al|\leq m} a_{ij,\al}(x) T^\al/\al! \in \cP$, where every $a_{ij,\al}$ is a 
semialgebraic function on $X$. Likewise, $B_x = (B_{1,x},\ldots,B_{k,x})$, where each $B_{i,x} = \sum b_{i,\al}(x) T^\al/\al!$, and every $b_{i,\al}$
is a semialgebraic function on $X$. So the equation \eqref{eq:eqjet} can be rewritten as
\begin{equation*}
\sum_{j=1}^D A_{ij,x} \odot F_{j,x} = B_{i,x},\quad i=1,\ldots,k,\quad x\in X.
\end{equation*}

\begin{theorem}\label{thm:mainjet}
Consider a system of equations \eqref{eq:eqjet}, where $X$ is a semialgebraic set and $\dim X =1$. Assume that \eqref{eq:eqjet}
has a solution $F$, where $F$ is a (vector-valued) $\cC^m$ Whitney field on $X$. Then \eqref{eq:eqjet}
has a solution $G$, where $G$ is a (vector-valued) $\cC^m$ semialgebraic Whitney field on $X$.
\end{theorem}

Note that, for every $x\in X$, the set of solutions of the associated homogeneous system of equations $A_x\odot F_x = 0$
(in other words, $\Ker A_x$) forms a $\cP$-submodule of $\cP^D$, since $\odot$ is commutative.

We will prove a more general result than Theorem \ref{thm:mainjet}, concerning a system of equations
\begin{equation}\label{eq:eq}
A(x)\cdot F_x = B(x),\quad x\in X,
\end{equation}
where $B: X\to \IR^k$ is a given $\IR^k$-valued semialgebraic function on $X$, $A$ is a given $k\times d$ matrix of semialgebraic
functions on $X$, the unknown $F=(F_x)_{x\in X}$ is a field of $\IR^D$-valued $m$-jets on $X$, and $d = D\binom{m+n}{m}$.

We will say that a system of equations \eqref{eq:eq} is \emph{multiplicative} 
if the linear subspace $\Ker A(x) \subset \cP^D$ is a $\cP$-submodule of $\cP^D$, for
all $x\in X$.

Using the notation above, each $F_x = (F_{1,x},\ldots,F_{D,x})$ can be identified with the vector of coefficients $(f_{j,\al}(x))$,
where $j=1,\ldots,D$ and $\al \in \IN^n$, $|\al| \leq m$. Therefore, writing $B = (b_1,\ldots,b_k)$ and
$A = (a_{i,(j,\al)})$, where $i$ and $(j,\al)$ are used as the row and column indices (respectively), the equation \eqref{eq:eq}
can be rewritten as
\begin{equation}\label{eq:eq1}
\sum_{j=1}^D \sum_{|\al|\leq m} a_{i,(j,\al)}(x) f_{j,\al}(x) = b_i(x),\quad \quad i=1,\ldots,k,\quad x\in X.
\end{equation}

\begin{theorem}\label{thm:main}
Consider a multiplicative system of equations \eqref{eq:eq}, where $X\subset\IR^n$ is a semialgebraic set and $\dim X =1$.
Assume that \eqref{eq:eq} has a solution $F$, where $F$ is a (vector-valued) $\cC^m$ Whitney field on $X$. Then
\eqref{eq:eq} has a solution $G$, where $G$ is a (vector-valued) $\cC^m$ semialgebraic Whitney field on $X$.
\end{theorem}

In this theorem and related results throughout the article, we write $\dim X =1$ (and ignore isolated points of $X$ in the
proofs) rather than $\dim X \leq 1$ because isolated points can always be accommodated in a trivial way using a 
$\cC^m$ semialgebraic partition of unity.

Theorem \ref{thm:mainjet} is a special case of Theorem \ref{thm:main}. Indeed, it follows from the formula for the product
$P\odot Q$ of $m$-jets $P,Q$, that a system of equations \eqref{eq:eqjet} can be rewritten as $A'(x)\cdot F_x = B'(x)$ in the form
of \eqref{eq:eq}, where $B': X \to \IR^\ell$, $\ell := k \binom{m+n}{m}$, and $A'(x) = (a_{ij}'(x))$, $i=1,\ldots,\ell$, $j=1,\ldots,d$,
and where each entry $a_{ij}'$ is either $a_{ih,\al}/\al!$ for some coefficient $a_{ih,\al}$ of an entry $A_{ih}$ of $A$, or zero.

Theorem \ref{thm:main} is not, in general, true without the multiplicative hypothesis on the system of equations \eqref{eq:eq}.

\begin{example}\label{ex:main}
Let $X = [0,1] \subset \IR$. Take $n=m=k=D=1$, $A(x)=\begin{pmatrix}0&1\end{pmatrix}$, and $B(x)= 1/(1+x^2)$.
Then \eqref{eq:eq} says that $f'(x)=1/(1+x^2)$, so that \eqref{eq:eq} admits $\arctan$ as a solution, but there is no $\cC^1$ semialgebraic solution.
\end{example}

There are also more general relative versions of the theorems above (and the corollaries below), where $F$ is
assumed to be semialgebraic on a subset $Y$ of $X$, and we find a semialgebraic solution $G$ such that $G|_Y = F|_Y$.
The relative versions are, in fact, corollaries of the results here, and the implications absolute $\implies$ relative statement do
not depend on $\dim X = 1$; see Section \ref{sec:plane}.

The following is a reformulation of Theorem \ref{thm:main} in terms of functions (as in the results of \cite{FL}) rather
than Whitney fields.

\begin{corollary}\label{cor:mainB}
With the hypotheses of Theorem \ref{thm:main}, assume that there is a $\cC^m$ (vector-valued) function $f$ on $\IR^n$
such that the Taylor field of $f$ on $X$ is a solution of \eqref{eq:eq}. Then there is a $\cC^m$ (vector-valued) semialgebraic function
$g$ on $\IR^n$ whose Taylor field on $X$ is a solution of \eqref{eq:eq}.
\end{corollary}

\begin{remarks}\label{rem:closure}
First, consider a closed subset $X\subset \IR^n$. 
Recall that Whitney's extension theorem asserts that $F$ is a $\cC^m$ (vector-valued) Whitney field on $X$
if and only if there is a $\cC^m$ (vector-valued) function $f$ on $\IR^n$ such that $F_x = J_x (f)$, $x\in X$ \cite{W}.
Kurdyka and Paw{\l}ucki's semialgebaic version of Whitney's extension theorem asserts that $F$ is a $\cC^m$ 
semialgebraic (vector-valued) Whitney field on $X$
if and only if there is a $\cC^m$ semialgebraic function $f$ on $\IR^n$ such that $F_x = J_x(f)$, $x\in X$ \cite{KP}.

Now consider an arbitrary subset $X\subset \IR^n$, and a system of equations \eqref{eq:eq} on $X$. Extend the system
of equations to the closure $\overline{X}$ by setting $A(x)=0$ and $B(x)=0$, for $x\in\overline X\setminus X$. Let $f$ denote
a $\cC^m$ (vector-valued) function on $\IR^n$. Then the Taylor field of $f$ on $X$ is a solution of \eqref{eq:eq} if and only if
the Taylor field of $f$ on $\overline{X}$ is a solution of the extended equation. The Taylor field of $f$ on $X$ is semialgebraic
if and only if the Taylor field on $\overline{X}$ is semialgebraic. Note that the system of equations on $X$ is multiplicative
if and only if the extended system on $\overline{X}$ is multiplicative. 
\end{remarks}

Therefore, in the case that $X$ is closed, Corollary \ref{cor:mainB} is equivalent to Theorem \ref{thm:main}, and
Corollary \ref{cor:mainB} follows for $X$ as in Theorem \ref{thm:main}, in general. On the other hand, our proof of Theorem \ref{thm:main}
in the case that $X$ is not closed involves particular properties of one-dimensional semialgebraic $X$; e.g., $X$ is
locally closed. 

Note that the assumption $\dim X =1$ is used to prove Theorem \ref{thm:main}, but all implications among the theorems
and corollaries in this section are proved for general semialgebraic $X$.

The following two results are consequences of Theorem \ref{thm:mainjet} (in the formulation of Corollary \ref{cor:mainB}). 
The first deals with equations of the form considered
in \cite{FK} and \cite{FL}, and the second with the semialgebraic Whitney extension problem.

\begin{corollary}\label{cor:mainM}
Let $X\subset\IR^n$ denote a semialgebraic set of dimension one. Let $A$ be a $k\times D$ matrix of semialgebraic functions on $X$, and let $B:X\to\IR^k$ be semialgebraic. Assume that the system of equations
\begin{equation}\label{eq:mainM}
A(x)\cdot f(x) = B(x),\quad x\in X,
\end{equation} 
has a solution $f\in\cC^m(\IR^n,\IR^D)$. Then \eqref{eq:mainM} has a $\cC^m$ semialgebraic solution.
\end{corollary}

\begin{proof}
Take $\ga\in\IN^n$ such that $|\ga| = m$.  Then, for every $P\in \cP$,
$T^\ga\odot P = P(0)T^\ga$. Define matrix- and vector-valued $m$-jet fields $A^\ga$ and $B^\ga$ on $X$ by
$A^\ga_x := A(x)T^\ga$ and $B^\ga_x := B(x)T^\ga$, $x\in X$. Then $A^\ga_x \odot P = B^\ga_x$ if and only if $A(x)\cdot P(0) = B(x)$.

Consider the system of equations
\begin{equation}\label{eq:eqB}
A^\ga_x \odot F_x = B^\ga_x,\quad x\in X,
\end{equation}
of the form \eqref{eq:eqjet}.
Since $J_x(f)(0) = f(x)$, the Taylor field $F_x := J_x(f)$ is a solution of \eqref{eq:eqB}. By Corollary \ref{cor:mainB},
there is a $\cC^m$ semialgebraic function $f \in \cC^m(\IR^n,\IR^D)$ such that $J_x(f)$ is a solution of \eqref{eq:eqB}, $x\in X$;
therefore, $A(x)\cdot f(x) = B(x)$ on $X$.

An analogous proof can be given using Theorem \ref{thm:main} directly.
\end{proof}

In view of Fefferman's general $\cC^m$ extension theorem \cite{F}, it makes sense to formulate a semialgebraic version
of the $\cC^m$ Whitney extension problem as follows.

\medskip\noindent
\emph{Semialgebraic Whitney extension problem.} Let $g: X\to \IR$ denote a semialgebraic function on a subset
$X$ of $\IR^n$, Suppose $g$ extends to a $\cC^m$ function $f$ on $\IR^n$. Can we find an extension of $g$ to a
$\cC^m$ semialgebraic function on $\IR^n$?

\begin{corollary}\label{cor:mainW1}
Consider a semialgebraic function $g:X\to\IR$ where $X\subset\IR^n$ is a semialgebraic set of dimension one.
If $g$ extends to a $\cC^m$ function $f$ on $\IR^n$, then $g$ extends to a $\cC^m$ semialgebraic function on $\IR^n$.
\end{corollary}

\begin{proof}
This follows from Corollary \ref{cor:mainM}, taking $k=D=1$, $A=(1)$ and $B(x)=g(x)$.
\end{proof}

In Section \ref{sec:plane}, we show how to deduce the following result of \cite{FL} from the relative version
Corollary \ref{cor:relmain} of Theorem \ref{thm:main}.
(The same argument, in fact, shows how to deduce this result
for $X\subset \IR^n$ from the assertion of Corollary \ref{cor:relmain} for $\dim X = n-1$.)

\begin{corollary}\label{cor:mainW2}
Consider a semialgebraic function $g:X\to\IR$ where $X\subset\IR^2$ is a semialgebraic set.
If $g$ extends to a $\cC^m$ function $f$ on $\IR^2$, then $g$ extends to a $\cC^m$ semialgebraic function on $\IR^2$.
\end{corollary}

\medskip\noindent
\emph{Bundle formulation.}
Fefferman and Luli reformulated the problem about $\cC^m$ semialgebraic solutions of equations as a problem about sections
of bundles. We do not need this formulation to prove Corollary \ref{cor:mainB} (or Theorem \ref{thm:main}), 
but we point out that the statements are equivalent.

\begin{definitions}\label{def:bundles}
Fix $X\subset \IR^n$. A $\cC^m(\IR^n,\IR^D)$-\emph{bundle} on $X$ is a family $\cH = (H(x))_{x\in X}$, where each $H(x)$
is a coset of a $\cP$-submodule of $\cP^D$.

We say that a bundle $\cH$ on $X$ is \emph{semialgebriaic} if $\{(x,P): x\in X,\, P\in H(x)\}$ is a semialgebraic subset
of $\IR^n\times \cP^D$. If $\cH$ is a semialgebraic bundle on $X$, then $X$ is semialgebraic, by the Tarski-Seidenberg
Theorem (since all fibres of $\cH$ are nonempty).

A \emph{section} of $\cH$ denotes a (vector-valued) function $f\in \cC^m(\IR^n,\IR^D)$ such that $J_x (f) \in H(x)$, for all $x\in X$.
\end{definitions}

A multiplicative system of equations \eqref{eq:eq} defines a semialgebraic bundle on $X$.
Conversely, given a semialgebraic bundle $\cH$ on a subset $X\subset \IR^n$, there is a multiplicative system of 
equations \eqref{eq:eq} which induces $\cH$
in this way; see Section \ref{sec:bundle}. Therefore, Corollary \ref{cor:mainB} and the following are, in fact, equivalent, without the
assumption $\dim X = 1$.

\begin{corollary}\label{cor:mainC}
Consider a semialgebraic bundle $\cH$ on $X \subset \IR^n$, where $\dim X =1$. Assume that $\cH$ has a section $f$. Then
$\cH$ has a semialgebraic section $g$.
\end{corollary}

\section{Proof of the main theorem}\label{sec:main}
Given a system of equations \eqref{eq:eq1} on a semialgebraic subset $X\subset \IR^n$, and a (vector-valued) $\cC^m$ Whitney field
$F = (F_x)_{x\in X}$ such that the vector of coefficients $(f_{j,\al}(x))$ is a solution of \eqref{eq:eq1}, we want to find a $\cC^m$
semialgebraic solution.
Introducing variables $y_{j,\al}$,
$j=1,\ldots,D$, $|\al|\leq m$, we get a system of semialgebraic equations $\sum_{j,\al} a_{i,(j,\al)}(x)\,y_{j,\al}  = b_i(x)$
in $(x,y) \in X \times \IR^d$, and can use semialgebraic selection (or the ``fibre-cutting lemma'') to find a semialgebraic solution
$y_{j,\al} = f_{j,\al}(x)$. This is, of course, not enough---we need to find semialgebraic functions $f_{j,\al}(x)$ satisfying \eqref{eq:eq1}
and also satisfying the conditions \eqref{eq:Whitney} which say that these functions define a Whitney field. But \eqref{eq:Whitney} is not a first
order condition, so it does not together with \eqref{eq:eq1} define a set to which we can apply semialgebraic selection.

Suppose $\dim X =1$. Then, locally, $X$ is a finite union of branches of curves. After a coordinate change, such a branch is
tangent to the $x_1$-axis at $0$, and can be written $x_\ell = \vp_\ell(x_1)$, $\ell=2,\ldots,n$, where $x_1 \in (0,\ep)$ or $[0,\ep)$. 
The functions $\vp_i(x_1),\, a_{i,(j,\al)}(x_1,\vp(x_1)),\, b_i(x_1,\vp(x_1))$ are all Nash functions
on $(0,\ep)$, after perhaps shrinking $\ep$.

The astuce of Fefferman and Luli is to show that a $\cC^m$ Whitney field $F$ on the given branch of $X$ 
(the graph of $\vp= (\vp_2,\ldots,\vp_n)$) satisfies \eqref{eq:eq1} if and only if the
functions $\tf_{j,\be}(x_1):= f_{j,(0,\be)}(x_1,\vp(x_1))$,
where $\be = (\be_2,\ldots,\be_n)$, $|\be|\leq m$, 
satisfy an associated system of semialgebraic equations 
\begin{equation}\label{eq:eq2}
\sum_{j=1}^D \sum_{|\be|\leq m} \ta_{i,(j,\be)}(x_1)\,\tf_{j,\be} = \tb_i(x_1), \quad i=1,\ldots,k(m+1),\quad x_1\in (0,\de),
\end{equation}
A semialgebraic Whitney field $G=(G_1,\ldots,G_D)$ as needed
can be realized as $G_j = \sum_\be \tf_{j,\be}(x_1) (y-\vp(x_1))^\be/\be!$, where $y = (x_2,\ldots,x_n)$, using
a semialgebraic section for the projection to $(0,\de)$ of a semialgebraic set which is defined by \eqref{eq:eq2} together with additional
first order conditions, not involving derivatives in $x_1$, but corresponding to Whitney estimates \eqref{eq:Whitney} only on $y$-derivatives
for pairs of points on adjacent tangent branches; see \eqref{eq:2nd}. 

The associated system of equations \eqref{eq:eq2} is constructed in the following lemma, inspired by \cite[Lemma 5.1]{FL}. We
use the notation of Section \ref{sec:intro}.

\begin{lemma}\label{lem:FL5.1}
Let $\psi:U\to\IR^{n-1}$ be a Nash function on an open subset $U\subset \IR$, and let $\Ga_\psi$ denote the graph of $\psi$; i.e.,
$\Gamma_\psi := \{(x,\psi(x))\in U\times\IR^{n-1}\}$. Consider a multiplicative system of equations \eqref{eq:eq} (or \eqref{eq:eq1})
for an $\IR^D$-valued $m$-jet field $F$ on $\Gamma_\psi$,
and assume that, for every $x\in U$, there exists $P\in\cP^D$ such that $A(x,\psi(x))\cdot P = B(x,\psi(x))$. 
Then there is a finite subset $S$ of $U$, and there are
semialgebraic functions $\ta_{i,(j,\beta)},\, \tb_i : U'\to\IR$ on the open set $U':= U\backslash S$, 
where $i=1,\ldots,k(m+1),\, j=1,\ldots,D$, and $\beta\in\IN^{n-1},\,|\beta|\le m$, such that, for every $\cC^m$ Whitney field $F$ on $\Gamma_\psi$ (over $U'$), 
\begin{equation}\label{eq:eqgraph}
A(x,\psi(x))F_{(x,\psi(x))}=B(x,\psi(x)),\quad \text{for all }\, x\in U',
\end{equation}
 if and only if
 \begin{equation}\label{eq:eqgraph1}
 \begin{split}
 \sum_{j=1}^D\sum_{\substack{\be\in\IN^{n-1}\\|\beta|\le m}}\ta_{i,(j,\beta)}(x)f_{j,(0,\beta)}(x,\psi(x))=\tb_i(x),\hspace{3cm}\\ 
 \text{ for all }\, i=1,\ldots,k(m+1) \text{ and } x\in U'.
 \end{split}
 \end{equation}
\end{lemma}

\begin{proof}
By semialgebraic selection, there is a semialgebraic function $\s: U\to\cP^D$ such that $A(x,\psi(x))\s(x)=B(x,\psi(x))$.
The functions $a_{i,(j,\alpha)}(x),\, B_i(x)$ and $\s(x)$ are Nash on the complement $U'$ of a finite subset $S$ of $U$.

Consider an $\IR^D$-valued $\cC^m$ Whitney field $F$ on $\Gamma_\psi$, $F=(F_1,\ldots,F_D)$. Then, for all $x\in U$,
$j=1,\ldots,D$ and $\al\in\IN^n$, $|\al|\leq m-1$,
\begin{equation}\label{eq:chain}
\p_x\left(f_{j,\al}(x,\psi(x))\right) = f_{j,\al +(1)}(x,\psi(x)) + \sum_{h=1}^{n-1} \frac{\p\psi_{h}(x)}{\p x} f_{j,\al +(h+1)}(x,\psi(x)),
\end{equation}
where $(x,y)=(x,y_1,\ldots,y_{n-1})\in U\times\IR^{n-1}$, and 
$(h)\in\IN^n$ denotes the multiindex with 1 in the $h$th place and 0 elsewhere, by the chain rule for Whitney fields
(or, alternatively, by the ordinary chain rule applied to the derivatives of 
a function $g\in \cC^m(V,\IR^D)$ on a neighbourhood $V$ of $\Gamma_\psi$, 
such that $F$ is the Taylor field of $g$ on $\Ga_\psi$).

Write $A(x,\psi(x))=(a_{i,(j,\alpha)}(x))$ and $B(x,\psi(x))=(B_1(x),\ldots,B_k(x))$. Let $L_{i,x}(P)$ denote the $i$th component of 
$A(x,\psi(x))P$, where $P\in\cP^D$; i.e.,
\begin{equation*}
L_{i,x}(P) = \sum_{j=1}^D\sum_{|\alpha|\le m} a_{i,(j,\alpha)}(x) \p^\alpha P_j(0);
\end{equation*}
thus $F$ is a solution of \eqref{eq:eqgraph} if and only if 
\begin{equation*}
L_{i,x}(F_{(x,\psi(x))}) :=\sum_{j=1}^D\sum_{|\alpha|\le m} a_{i,(j,\alpha)}(x) f_{j,\al}(x,\psi(x)) =B_i(x), 
\end{equation*}
for all $i=1,\ldots,k$ and $x\in U'$. 

By \eqref{eq:chain} and Leibniz's differentiation rule, we obtain
\begin{equation}\label{eq:eq*}
L_{i,x}(F_{(x,\psi(x))}) = \sum_{\ell=0}^m \p_x^\ell \Lambda_{i\ell}(F)(x),
\end{equation}
where
\begin{equation*}
\Lambda_{i\ell}(F)(x) := \sum_{j=1}^D\sum_{\substack{\be\in\IN^{n-1}\\|\beta|\le m-\ell}} C_{i \ell j \beta}(x) f_{j,(0,\be)}(x,\psi(x)),
\end{equation*}
and the $C_{i \ell j \beta}$ are Nash functions on $U'$.

Therefore, $F$ is a solution of \eqref{eq:eqgraph} if and only if $\sum_{\ell} \p_x^\ell \Lambda_{i\ell}(F) = B_i$, $i=1,\ldots,k$. Our aim is 
to find an equivalent system of equations of the form $ \Lambda_{i\ell}(F) = \gamma_{i\ell},\, i=1,\ldots,k,\, \ell=0,\ldots,m$, 
where the $\gamma_{i\ell}$ are Nash functions.

Fix $s\in\{0,\ldots,m\}$ and $x_0\in U'$, and set $p_0 := (x_0,\psi(x_0))$, 
$\chi(x,y) := \chi(x) = (x-x_0)^s / s!$. Then $J_{p_0}(\chi)(T)= T_1^s/s!$. We claim that, for every solution $F$ of \eqref{eq:eqgraph},
\begin{equation}\label{eq:eq**}
\sum_{\ell=s}^m\binom{\ell}{s} \Lambda_{i\ell}(F)^{(\ell-s)}(x_0)=c_{is}(x_0),\ \ \text{where}\ \ 
c_{is}(x_0) :=  L_{i,x_0}\left(\frac{T_1^s}{s!}\odot\s(x_0)\right).
\end{equation}

Indeed, $F_{p_0} -\s(x_0)\in \Ker A(p_0)$, and the latter is a submodule, so that
\begin{equation*}
J_{p_0}(\chi)\odot F_{p_0}-\frac{T_1^s}{s!}\odot\s(x_0) = J_{p_0}(\chi)\odot\left(F_{p_0}-\s(x_0)\right)\in \Ker A(p_0);
\end{equation*}
hence $L_{i, x_0} \left(J_{p_0}(\chi)\odot F_{p_0}\right)=c_{is}(x_0)$. Since $\chi$ is independent of $y$, 
$\Lambda_{i\ell}\left(J_{p_0}(\chi)\odot F_{p_0}\right)\allowbreak =\chi\Lambda_{i\ell}(F)$. 
Therefore evaluating \eqref{eq:eq*} for $J(\chi) \odot F$ at $p_0$ (where $J(\chi)_p := J_p(\chi)$) gives
\begin{equation*}
c_{is}(x_0) = \sum_{\ell=0}^m \p_x^\ell \left(\chi\Lambda_{i\ell}(F)\right)(x_0) 
= \sum_{\ell=s}^m\binom {\ell}{s} \Lambda_{i\ell}(F)^{(\ell-s)}(x_0),
\end{equation*}
as required.

Note that $c_{is}$ is Nash, since it is a polynomial in the $a_{i,(j,\alpha)}(x)$ and $\s(x)$. Moreover, $\chi\equiv1$ for $s=0$; therefore, $c_{i0}=B_i$.

We now define the following Nash functions $\ga_{is}$ recursively, from $s=m$ to $0$.
\begin{equation}\label{eq:ind}
\gamma_{is} := c_{is} - \sum_{\ell=s+1}^m \binom{\ell}{s} \gamma_{i\ell}^{(\ell-s)}.
\end{equation}

By \eqref{eq:eq**}, if $F$ is a solution of \eqref{eq:eqgraph}, then $\Lambda_{is}(F)=\gamma_{is}$, for all $i$ and $s$.

Conversely, assume that $\Lambda_{i\ell}(F)=\gamma_{i\ell}$, for all $i$ and $\ell$.
Since $c_{i0}=B_i$, it follows from \eqref{eq:ind} that $\sum_{\ell=0}^m \gamma_{i\ell}^{(\ell)}=B_i$, so that \eqref{eq:eq*}
implies $L_{i,x}(F_{(x,\psi(x))})=\sum_{\ell=0}^m \gamma_{i\ell}^{(\ell)}=B_i$: i.e. $F$ is a solution of \eqref{eq:eqgraph}

We get the assertion of the lemma by taking  $\ta_{(i,\ell),(j,\be)}:=C_{i\ell j\be}$ and $\tb_{i\ell}:=\gamma_{i\ell}$, where
$i=1,\ldots,k$ and $\ell=0,\ldots,m$. (The $\ta_{(i,\ell),(j,\be)}$ are the entries of a matrix with row indices $(i,\ell)$ and
column indices $(j,\be)$.)
\end{proof}

\begin{remarks}\label{rem:FL5.1} (1) Note that $c_{is}$ in the proof does not depend on the choice of $\s$ (though we do
not use this remark):
if $A(p_0)\tilde\s(x_0)=B(p_0)$, then $\s(x_0)-\tilde\s(x_0)\in \Ker A(p_0)$, hence $(T_1^s/s!)\odot(\s(x_0)-\tilde\s(x_0))\in \Ker A(p_0)$. 

\medskip\noindent
(2) The proof above can be used for a more general version of Lemma \ref{lem:FL5.1}, where $\psi:U\to\IR^q$ is a Nash 
function on an open subset $U\subset \IR^p$, and $S$ is a closed semialgebraic subset of $U$ of dimension $<p$
(cf. \cite[Lemma 5.1]{FL}).

\medskip\noindent
(3) In each of the theorems and corollaries in Section \ref{sec:intro} involving a system of equations \eqref{eq:eq}
(or a special case of the latter), the stated assumption implies that the system of equations admits a solution $P\in \cP^D$
at every point.
\end{remarks}

\begin{proof}[Proof of Theorem \ref{thm:main}]
Let $X$ be a 1-dimensional semialgebraic subset of $\IR^n$. Then $X$ is locally closed (i.e., $X$ is closed in some
open neighbourhood) and $\overline{X}\backslash X$ is a finite set.

There is a finite subset $\Sigma$ of $X$ such that $X\backslash \Sigma$ is Nash submanifold of $\IR^n$ of pure dimension $1$,
Therefore, there is a perhaps larger finite subset $T \supset \Sigma$ of $X$ such that, for every connected component $C$
of $X\backslash T$, the mapping $C\to \IR$ given by the projection $(x_1,\ldots,x_n)\mapsto x_i$ has constant rank $1$,
for some $i$. It follows that $X\backslash T$ is a finite disjoint union of the graphs $\Ga_\psi$ of Nash functions $\psi: U \to \IR^{n-1}$,
where $U$ is an open subset of the $x_i$-axis, for some $i$, with values in $\IR^{n-1}$ (with coordinates $(x_1,\ldots,\widehat{x_i},\ldots x_n)$,
$x_i$ removed).

Consider a mutiplicative system of equations \eqref{eq:eq} on $X$, which admits a $\cC^m$ Whitney field solution. Using Lemma \ref{lem:FL5.1},
we will find a finite covering of $X$ by open semialgebraic sets $V_\rho$, such that, for each $\rho$, $X\cap V_\rho$ is closed in $V_\rho$ 
and the conclusion of Theorem \ref{thm:main} holds for $X\cap V_\rho$. The conclusion of the theorem for $X$ then follows 
using a $\cC^m$
semialgebraic partition of unity for $X$ subordinate to the covering $\{V_\rho\}$ (the multiplicative property of \eqref{eq:eq} is
needed here). 

First, by Lemma \ref{lem:FL5.1}, after enlarging the finite subset $T$ of $X$, we can assume that, for each of the Nash functions $\psi$ above,
a $\cC^m$ Whitney field $F$ on $\Ga_\psi$ is a solution of \eqref{eq:eq} on $\Ga_\psi$ if and only if
the coefficients $f_{j,(0,\be)}(x,\psi(x))$, $x\in U$, satisfy the system of equations \eqref{eq:eqgraph1}. So the
existence of a $\cC^m$ Whitney field solution of \eqref{eq:eq} implies, by semialgebraic selection, that
\eqref{eq:eqgraph1} has a semialgebraic solution $f_{j,\be}(x)$, $x\in U$. 
By enlarging $T$ again by finitely many additional points, we can assume that every 
$f_{j,\be}(x)$ is a Nash function on $U$ (for every $\psi$). Define $f' = (f_1',\ldots,f_D'): U\times \IR^{n-1} \to \IR^D$ by

\begin{equation*}
f'_j(x,y) := \sum_{|\be|\leq m} \frac{f_{j,\be}(x)}{\be!} (y-\psi(x))^\be,\quad j=1,\ldots.D.
\end{equation*}
Then $f'$ is a vector-valued Nash function, and its Taylor field $G := (J_p (f'))_{p\in \Ga_\psi}$ is a $\cC^m$ semialgebraic Whitney field
which is solution of $\eqref{eq:eq}$ on $\Ga_\psi$, again by Lemma \ref{lem:FL5.1}.

Note that $V_\psi := \{z\in \IR^n: \mathrm{d}(z,\Ga_\psi) < \mathrm{d}(z, \overline{X\backslash \Ga_\psi})\}$, 
is a semialgebraic open neighbourhood of $\Ga_\psi$ containing no other points of $X$. There are finitely many $\Ga_\psi \subset X$,
and $X = T \bigcup_\psi \Ga_\psi$.

So it remains to show that every point $p\in T$ has a semialgebraic neighbourhood $V_p$ such that the conclusion of Theorem
holds for $X\cap V_p$. This assertion is trivial if $p$ is an isolated point of $X$. 

Note that, up to this point in the proof, we have
used that assumption that the system of equations \eqref{eq:eq} has a $\cC^m$ Whitney field solution only to use Lemma \ref{lem:FL5.1}.

Let $B(p,\ep)$ denote the open ball with centre $p$ and radius $\ep > 0$. Suppose $p\in T$ is not an isolated point of $X$.
Then, for $\ep$ small enough, $X\cap B(p,\ep)$ is a finite union of continuous semialgebraic curves $\ga: [0,\de_\ga) \to B(p,\ep)$
such that $\ga(0) = p$. We can assume that each $\ga$ is Nash on $(0,\de_\ga)$, and the $\ga\backslash\{p\}$ are disjoint.

Let us group the curves $\ga$ into clusters $\cK$, where $\cK$ denotes the union of the curves $\ga$ with the same tangent direction at $p$
(the tangent direction is $\lim_{t\to 0+} \ga'(t)/|\ga'(t)|$). 

Then any two clusters $\cK_1,\,\cK_2$ are $1$-\emph{regularly situated} at $p$; 
i.e., there exists $c>0$ such that $\mathrm{d}(x,\cK_1) + \mathrm{d}(x,\cK_2) \geq c\, \mathrm{d}(x,\cK_1\cap \cK_2)$ in some neighbourhood of $p$.
Therefore, by {\L}ojasiewicz's gluing theorem \cite[{\S}I.5]{M}, given $\cC^m$ Whitney fields $G_1$ on $\cK_1$ and $G_2$ on $\cK_2$,
with $G_{1,p} = G_{2,p}$, there is a $\cC^m$ Whitney field $G$ on $\cK_1 \cup \cK_2$ such that $G=G_1$ on $\cK_1$ and $G=G_2$ on $\cK_2$.

So it is enough to prove the theorem for each cluster, provided we reduce to the case that $F$ is flat at $p$ (i.e., $F_p = 0$) and show that,
in this case, we can find a $\cC^m$ semialgebraic solution which is also flat at $p$. We can indeed assume that $F$ is flat at $p$ because
$F_p \in \cP^D$ can be considered a polynomial mapping on $\IR^n$, so it suffices to prove the theorem for $F-G$, where $G$ is the
Whitney field given by $G_x := J_x(F_p)$, for all $x$.

We can assume that $p=0$. For a given cluster $\cK$, we can also assume that the common tangent direction is the positive $x_1$-direction
(by an orthogonal transformation).
In conclusion, we have reduced the remaining special case of the theorem to the following proposition.
\end{proof}

\begin{proposition}\label{prop:cluster}
Let $\psi_1,\ldots,\psi_r:[0,\de)\to\IR^{n-1}$ denote continuous semialgebraic functions such that,
for every $x\in(0,\de)$, $\psi_1(x),\ldots,\psi_r(x)$ are distinct, and, for every $s=1,\ldots,r$,
\begin{equation*}
\psi_s \text{ is Nash on } (0,\de),\qquad \psi_s(0)=0,\qquad \lim_{x\to 0+}\psi_s'(x)=0.
\end{equation*}
Set $X := \bigcup_{s=1}^r\Gamma_{\psi_s}\subset\IR^n$, and consider a multiplicative system of equations \eqref{eq:eq} on $X$.
Assume that there exists a $\cC^m$ Whitney field $F$ on $B(0,\de)\cap X$ such that $F_0=0$ and $A(p)\cdot F_p=B(p)$, for all
$p\in B(0,\de)\cap X$. Then there exists a $\cC^m$ semialgebraic Whitney field $G$ on
$B(0,\ep)\cap X$, where $0<\ep\leq \de$, such that $G_0 =0$ and $A(p)\cdot G_p=B(p)$, for all $p\in B(0,\ep)\cap X$.
\end{proposition}

\begin{proof}
Set 
\begin{equation*}                    
\rho_{st}(x):= |\psi_s(x)-\psi_t(x)|,\quad s,t=1,\ldots,r,
\end{equation*}
and 
$\rho(x):=\max_{s}|\psi_s(x)|$. We can assume that $\rho(x)\leq x$, $x\in(0,\ep)$, by decreasing $\ep$ if necessary.
 
By Lemma \ref{lem:FL5.1} (decreasing $\ep$ again if needed), there is a finite index set $I$ and, for each $s=1,\ldots,r$,
there are semialgebraic functions $\ta^s_{i,(j,\beta)}(x),\, \tb^s_i(x)$, $x\in (0,\ep)$, where $i\in I$, $j=1,\ldots,D$, and $\be\in\IN^{n-1}$, $|\be|\leq m$,
such that, for any $\cC^m$ Whitney field $H$ on $\Gamma_{\psi_s}$,
\begin{equation*}
A(x,\psi_s(x))H_{(x,\psi_s(x))}=B(x,\psi_s(x)),\quad \text{for all }\, x\in (0,\ep),
\end{equation*}
 if and only if
\begin{equation}\label{eq:eqn}
\sum_{j=1}^D\sum_{|\beta|\le m}\ta^s_{i,(j,\beta)}(x)h_{j,(0,\beta)}(x,\psi_s(x))=\tb^s_i(x),\quad \text{for all } i\in I \text{ and } x\in (0,\ep).
\end{equation}

Set $f^s_{j,\be}(x):=f_{j,(0,\beta)}(x,\psi_s(x))$, $x\in (0,\ep)$, $s=1,\ldots,r$. These functions satisfy the following three conditions:
\begin{align}
&\sum_{j=1}^D \sum_{|\beta|\leq m}\ta^s_{i,(j,\beta)}(x)f^s_{j,\be}(x)=\tb^s_i(x),\quad i\in I,\,\,s=1,\ldots,r;\label{eq:1st}\\
&f^s_{j,\al}(x)=\sum_{|\beta|\leq m-|\alpha|}\frac{f^t_{j,\alpha+\beta}(x)}{\beta!}\left(\psi_s(x)-\psi_t(x)\right)^\beta+\o{x\to 0+}\left(\rho_{st}(x)^{m-|\alpha|}\right),\label{eq:2nd}\\ 
&\hspace{5cm} \al, \be \in\IN^{n-1},\,\, s,t=1,\ldots,r,\,\,s\neq t;\notag\\
&f^s_{j,\al}(x)=\o{x\to0+}\left(x^{m-|\alpha|}\right),\quad s=1,\ldots,r.\label{eq:3rd}
\end{align}
The condition \eqref{eq:1st} is from \eqref{eq:eqn} (i.e., from Lemma \ref{lem:FL5.1}),
\eqref{eq:2nd} is from the Whitney field condition \eqref{eq:Whitney}
for a pair of points $p=(x,\psi_s(x)),\, q=(x,\psi_t(x))$ tending to $0$,
and \eqref{eq:3rd} is from the latter condition where $q=0$, together with the flatness assumption $F_0 = 0$
and the inequalities $x\leq |p|\leq \sqrt{2}x$ (from the assumption $\rho(x)\leq x$).

Consider the mapping  $\Phi: (0,\ep) \to \IR^N$ with components $f^s_{j,\be}(x)$, in any given order. 
Condition \eqref{eq:1st} defines a semialgebraic subset $V$ of $(0,\ep)\times\IR^N$
such that $\Phi(x)\in V_x$, for all $x\in (0,\ep)$, where $V_x$ denotes the fibre of $V$ over $x$.
Dividing \eqref{eq:2nd} and
\eqref{eq:3rd} by $\rho_{st}(x)^{m-|\alpha|}$ and $x^{m-|\alpha|}$, respectively, we obtain a matrix-valued semialgebraic function $\La(x)$
such that $\La(x)\cdot\Phi(x) \to 0$ as $x\to 0+$.  Since $\lim_{x\to 0+}\Phi(x)=0$, the origin belongs to 
the closure of semialgebraic set $\{(x,\La(x)\cdot v): v\in V_x\}$. By the curve selection lemma, there is a semialgebraic curve in this set
tending to $0$. By semialgebraic selection, we then obtain semialgebraic functions $f^s_{j,\be}: [0,\ep)\to \IR$ with all $f^s_{j,\be}(0)=0$,
satisfying the conditions \eqref{eq:1st}--\eqref{eq:3rd}.

\emph{From now on, $f^s_{j,\be}$ will denote these semialgebraic functions, rather than the functions we started with.}
We can assume that all $f^s_{j,\be}$ are Nash on $(0,\ep)$.

We recall that a Nash function $\eta(x)$ on $(0,\ep)$ has a Puiseux expansion, so that 
\begin{equation}\label{eq:puis}
\p_x^k \,\eta(x) = \O{x\to 0+}\left(x^{-k}\eta(x)\right).
\end{equation}

Set
\begin{equation}\label{eq:fs}
f^s_j(x,y) := \sum_{\substack{\be\in\IN^{n-1}\\|\be|\leq m}} \frac{f^s_{j,\be}(x)}{\be!} (y-\psi_s(x))^\be, \quad j=1,\ldots.D,\, s=1,\ldots,r,
\end{equation}
where $x\in [0,\ep)$ and $y\in\IR^{n-1}$. The following three claims establish properties of these functions that we need.
In these claims, we fix $j$. Set $E_\ep := \{(x,y):\,x\in[0,\ep),\, \allowbreak |y|\leq x\}$.

\begin{claim}\label{claim:first}
\begin{equation*}
\p^{(\ell,\alpha)}f^s_j(x,y)=\o{(x,y)\to0}\left(x^{m-|\alpha| -\ell}\right),\quad s=1,\ldots,r,\ \ell\in\IN,\ \alpha\in\IN^{n-1},
\end{equation*}
where $(x,y)\in E_\ep$. (We do not assume that $\ell +|\al| \leq m$.)
\end{claim}

\begin{proof}
Using \eqref{eq:puis} and $\psi_s(x)=\o{}(x)$, we see that
\begin{equation*}
\partial_x^p\left((y-\psi_s(x))^{\beta-\alpha}\right)=\O{}(x^{|\beta|-|\alpha|-p})\quad \text{ on }\, E_\ep,
\end{equation*}
and, therefore, using also \eqref{eq:2nd}, \eqref{eq:3rd},
\begin{equation*}
\p^{(\ell,\alpha)}f^s_j(x,y)=\sum_{\substack{\beta\geq\alpha\\|\be|\leq m}}
\sum_{q=0}^{\ell} \O{}\left(x^{|\beta|-|\alpha|-(\ell-q)}\right) \o{}\left(x^{m-|\beta|-q}\right)\quad \text{ on }\, E_\ep.
\end{equation*}
The result follows.
\end{proof}

\begin{claim}\label{claim:second}
For all $s,t=1,\ldots,r$ with $s\neq t$, and all $\ell\in\IN$, $\alpha\in\IN^{n-1}$, where $\ell+|\alpha|\leq m$, 
\begin{equation*}
\begin{split}
\p&^{(\ell,\alpha)}f^s_j(x,\psi_s(x))\\
&=\sum_{|\beta|\le m-|\alpha|-\ell} \frac{\left(\psi_s(x)-\psi_t(x)\right)^\beta}{\beta!} \p^{(\ell,\alpha+\beta)}f^t_j\left(x,\psi_t(x)\right) 
+\o{x\to0+}\left(\rho_{st}(x)^{m-|\alpha|-\ell}\right).
\end{split}
\end{equation*}
\end{claim}

\begin{proof}
We argue by induction on $\ell$. The base case $\ell=0$ is immediate from \eqref{eq:2nd}.  Note that the claim for given $\ell$
is equivalent to the statement that 
\begin{equation}\label{eq:toprove}
W_{\ell,\alpha}(x) = \o{x\to0+}\left(\rho_{st}(x)^{m-|\alpha|-\ell}\right), 
\end{equation}
where 
\begin{equation*}
W_{\ell,\alpha}(x) := \p^{(\ell,\alpha)}f^s_j(x,\psi_s(x))-\sum_{|\beta|\le m-|\alpha|-\ell-1}\frac{\left(\psi_s(x)
-\psi_t(x)\right)^\beta}{\beta!} \p^{(\ell,\alpha+\beta)}f^t_j\left(x,\psi_t(x)\right),
\end{equation*}
because the additional term of the sum in Claim \ref{claim:second} is $\o{x\to0+}\left(\rho_{st}(x)^{m-|\alpha|-\ell}\right)$, by Claim \ref{claim:first}.
So let us assume \eqref{eq:toprove} and prove the same statement for $\ell +1$.

We first compute the derivative $W'_{\ell,\alpha}(x)$, and compare with $W_{\ell+1,\alpha}(x)$ to find
\begin{align*}
W&_{\ell+1,\alpha}(x) = W'_{\ell,\alpha}(x) + \sum_{|\beta|=m-|\alpha|-\ell-1}\frac{\left(\psi_s(x)-\psi_t(x)\right)^\beta}{\beta!}\p^{(\ell+1,\alpha+\beta)}f^t_j(x,\psi_t(x))\\
&-\sum_{i=1}^{n-1} \p^{(\ell,\alpha +(i))}f^s_j(x,\psi_s(x))\psi'_{s,i}(x) \\
&+\sum_{|\beta|\le m-|\alpha|-\ell-1}\, \sum_{i=1}^{n-1} \frac{\left(\psi_s(x)-\psi_t(x)\right)^\beta}{\beta!} \p^{(\ell,\alpha+\beta+(i))}f^t_j\left(x,\psi_t(x)\right)\psi'_{t,i}(x)\\
&+\sum_{|\beta|\le m-|\alpha|-\ell-1} \ \sum_{\substack{i=1,\ldots,n-1\\ \be_i\neq 0}}\frac{\left(\psi_s(x)-\psi_t(x)\right)^{\beta-(i)}}{(\beta-(i))!} \left(\psi'_{s,i}(x)-\psi'_{t,i}(x)\right)
 \p^{(\ell,\alpha+\beta)}f^t_j\left(x,\psi_t(x)\right),
\end{align*}\\
so that
\begin{align*}
W_{\ell+1,\alpha}(x) = &W'_{\ell,\alpha}(x) + \sum_{|\beta|=m-|\alpha|-\ell-1}\frac{\left(\psi_s(x)-\psi_t(x)\right)^\beta}{\beta!}\p^{(\ell+1,\alpha+\beta)}f^t_j(x,\psi_t(x))\\
&-\sum_{i=1}^{n-1}\psi'_{s,i}(x) W_{\ell,\al+(i)}\\
&+ \sum_{i=1}^{n-1}\psi'_{t,i}(x) \sum_{|\beta|=m-|\alpha|-\ell-1} \frac{\left(\psi_s(x)-\psi_t(x)\right)^\beta}{\beta!}
\p^{(\ell,\alpha+\beta+(i))}f^t_j(x,\psi_t(x)).
\end{align*}\\

We can separately estimate the order of each of the preceding four summands of $W_{\ell+1,\alpha}(x)$.

\medskip\noindent
(1) By \eqref{eq:puis},
\begin{equation}\label{eq:(1)}
W'_{\ell,\alpha}=\o{x\to0^+}\left(\rho_{st}(x)^{m-|\alpha|-\ell}x^{-1}\right) = \o{x\to0+}\left(\rho_{st}(x)^{m-|\alpha|-\ell-1}\right).
\end{equation}

\medskip\noindent
(2) When $|\beta|=m-|\alpha|-\ell-1$, $\p^{(\ell+1,\alpha+\beta)}f^t_j\left(x,\psi_t(x)\right)=\o{x\to0+}\left(1\right)$, by Claim \ref{claim:first},
so that 
\begin{equation*}
\frac{\left(\psi_s(x)-\psi_t(x)\right)^\beta}{\beta!}\p^{(\ell+1,\alpha+\beta)}f^t_j(x,\psi_t(x)) 
= \o{x\to0+}\left(\rho_{st}(x)^{m-|\alpha|-\ell-1}\right).
\end{equation*}

\medskip\noindent
(3) By the inductive assumpion, since $\psi_{s,i}'(x)=\O{x\to0+}(1)$, 
\begin{equation*}
\sum_{i=1}^{n-1}\psi_{s,i}'(x)W_{\ell,\alpha+(i)}(x)=\o{x\to0+}\left(\rho_{st}(x)^{m-|\alpha|-\ell-1}\right).
\end{equation*}

\medskip\noindent
(4) When $|\beta|=m-|\alpha|-\ell-1$, we get
\begin{equation*}
\psi'_{t,i}(x) \frac{\left(\psi_s(x)-\psi_t(x)\right)^\beta}{\beta!}
\p^{(\ell,\alpha+\beta+(i))}f^t_j(x,\psi_t(x)) = \o{x\to0+}\left(\rho_{st}(x)^{m-|\alpha|-\ell-1}\right)
\end{equation*}
as in (2), since $\psi_{t,i}'(x)=\O{x\to0+}(1)$.\\

This completes the inductive step, to prove Claim \ref{claim:second}.
\end{proof}

\begin{claim}\label{claim:third}
For all $s,t=1,\ldots,r$ with $s\neq t$, and all $\ell\in\IN$, $\alpha\in\IN^{n-1}$, where $\ell+|\alpha|\leq m$, 
\begin{equation*}
\p^{(\ell,\alpha)}\left(f^s_j-f^t_j\right)(x,\psi_s(x))=\o{x\to0+}\left(\rho_{st}(x)^{m-|\alpha|-\ell}\right).
\end{equation*}
\end{claim}

\begin{proof}
We write 
\begin{equation*}
\begin{split}
\p^{(\ell,\alpha)}&\left(f^s_j-f^t_j\right)(x,\psi_s(x))\\
&=\left(\p^{(\ell,\alpha)}f^s_j(x,\psi_s(x))-P_x(\psi_s(x))\right)-\left(\p^{(\ell,\alpha)} f^t_j(x,\psi_s(x))-P_x(\psi_s(x))\right),
\end{split}
\end{equation*}
where 
\begin{equation*}
P_x(y) := \sum_{\substack{\beta\in\IN^{n-1}\\|\beta|\le m-|\alpha|-\ell}}\frac{\p^{(\ell,\alpha+\beta)}f^t_j\left(x,\psi_t(x)\right)}{\beta!} \left(y-\psi_t(x)\right)^\beta;
\end{equation*}
i.e., $P_x(y)$ is the Taylor polynomial of order $m-|\alpha|-\ell$ of $\p^{(\ell,\alpha)}f^t_j(x,\cdot)$ at $\psi_t(x)$.

By Claim \ref{claim:second},
\begin{equation*}
\p^{(\ell,\alpha)}f^s_j(x,\psi_s(x))-P_x(\psi_s(x))=\o{x\to0+}\left(\rho_{st}(x)^{m-|\alpha|-\ell}\right).
\end{equation*}

Since $\p^{(\ell,\alpha)}f^t_j(x,y)$ is a polynomial of degree $\leq m-|\alpha|$ in $y$, it coincides with its Taylor expansion 
of order $m-|\alpha|$ at $\psi_t(x)$; therefore,
\begin{equation*}
\begin{split}
\p^{(\ell,\alpha)}f^t_j(x,&\psi_s(x))-P_x(\psi_s(x))\\
&=\sum_{m-|\alpha|-\ell<|\beta|\leq m-|\alpha|}
\frac{\p^{(\ell,\alpha+\beta)}f^t_j(x,\psi_t(x))}{\beta!} \left(\psi_s(x)-\psi_t(x)\right)^\beta.
\end{split}
\end{equation*}

In each summand, $\left(\psi_s(x)-\psi_t(x)\right)^\beta=\O{x\to0+}\left(\rho_{st}(x)^{|\beta|}\right)$ and, by Claim \ref{claim:first},
\begin{equation*}
\p^{(\ell,\alpha+\beta)}f^t_j(x,\psi_t(x))=\o{x\to0+}\left(x^{m-|\alpha|-l-|\beta|}\right).
\end{equation*}
So each summand can be estimated as $\o{x\to0+}\left(\rho_{st}(x)^{|\beta|}x^{m-|\alpha|-\ell-|\beta|}\right)$
and also, therefore, as $\o{x\to0+}\left(\rho_{st}(x)^{m-|\alpha|-\ell}\right)$,
since $\rho_{st}(x)\leq 2\rho(x)\le 2x$ and $|\beta|>m-|\alpha|-\ell$.
This proves Claim \ref{claim:third}.
\end{proof}

We can now complete the proof of Proposition \ref{prop:cluster}. For each $s=1,\ldots,r$, let
\begin{equation*}
f^s(x,y) := (f^s_1(x,y),\ldots,f^s_D(x,y)),\quad (x,y) \in (0,\ep) \times \IR^{n-1}
\end{equation*}
(see \eqref{eq:fs}). Then $f^s(x,y)$ is an $\IR^D$-valued Nash function on $(0,\ep) \times \IR^{n-1}$, and all
partial derivatives $\p^{(\ell,\al)}f^s(x,y)$, where $\ell\in\IN,\, \al\in\IN^{n-1},\, \ell + |\al| \leq m$, tend to zero as
$(x,y) \to 0$ in the interior of $E_\ep$, by Claim \ref{claim:first}. Therefore, $f^s(x,y)$, $(x,y)\in E_\ep \cap \{x>0\}$,
extends to an $\IR^D$-valued $\cC^m$ semialgebraic function $g^s(x,y)$ on $E_\ep$, which is $m$-flat at $(0,0)$
(because $E_\ep$ is \emph{$1$-regular}; i.e., any two points $p,q \in E_\ep$ can be joined by a curve of 
length $\leq c\,|p-q|$ in $E_\ep$; here $c =1$. Cf. \cite{W1}.)

Let $G^s$ denote the $m$-jet field of $g^s$ on $\overline{\Ga_{\psi_s}}$. Then $G^s$ is an $\IR^D$-valued
semialgebraic $\cC^m$ Whitney field such that $G^s_0 = 0$. Let $G$ denote
the $m$-jet field on $X \cap B(0,\ep)$ given by $G^s$ on each curve $\overline{\Ga_{\psi_s}}$. 
Then $A(p)\cdot G_p = B(p)$, $p\in X\cap B(0,\ep)$, by \eqref{eq:1st}. 

Say $G_p = \sum_{\al\in\IN^n, |\al|\leq m} g_\al(p) T^\al/\al! \in \cP^D$, $p\in X \cap B(0,\ep)$.
It is enough to show that, for all $\al\in\IN^n$, $|\al|\leq m$,
\begin{equation}\label{eq:Wh}
g_\al(p) - \sum_{\substack{\be\in \IN^n \\ |\be| \leq m-|\al|}} g_{\al +\be}(q) \frac{(p-q)^\be}{\be!} = \o{}(|p-q|^{m-|\al|})
\end{equation}
as $p,q \to 0$ in $X\backslash \{0\}$ (cf. \eqref{eq:Whitney}).

Write $p=(x,\psi_s(x))$, $q=(x',\psi_t(x'))$, where $x, x' \in (0,\ep)$. Then the left-hand side of \eqref{eq:Wh}
can be written
\begin{equation*}
\p^\al(g^s - g^t)(p) + \left(\p^\al g^t(p) - \sum_{|\be| \leq m-|\al|} \p^{\al +\be}g^t(q) \frac{(p-q)^\be}{\be!}\right).
\end{equation*}
The second term of this expression is $\o{}(|p-q|^{m-|\al|})$, since $g^t$ is $\cC^m$, while the first term
is $\o{}(\rho_{st}(x)^{m-|\al|})$, by Claim \ref{claim:third}. But
\begin{align*}
\rho_{st}(x) & \leq |\psi_s(x) - \psi_t(x')| + |\psi_t(x') - \psi_t(x)|\\
& \leq |p-q| + c |x-x'| \leq (1+c)|p-q|,
\end{align*}
where $c := \max_{u=1,\ldots,r} \sup_{(0,\ep)}|\psi'_u|$, so the result follows.
\end{proof}

\section{Relative statements and semialgebraic Whitney extension in the plane}\label{sec:plane}
We begin with a relative version of Theorem \ref{thm:main}, and afterwards show how it can be useful, in a proof
of Corollary \ref{cor:mainW2} on the $\cC^m$ semialgebraic Whitney extension problem in $\IR^2$.

\begin{corollary}\label{cor:relmain}
Consider a multiplicative system of equations \eqref{eq:eq}, where $X\subset\IR^n$ is a semialgebraic set and $\dim X =1$.
Assume that \eqref{eq:eq} has a solution $F$, where $F$ is a (vector-valued) $\cC^m$ Whitney field on $X$, and $F$ is
semialgebraic on a subset $Y$ of $X$. Then
\eqref{eq:eq} has a solution $G$, where $G$ is a (vector-valued) $\cC^m$ semialgebraic Whitney field on $X$, and $G|_Y = F|_Y$.
\end{corollary}

\begin{proof}
Recall that we identify the vector-valued $m$-jet $F_x\in \cP^D$, $x\in X$, with the vector of coefficients 
$(f_{j,\al}(x))_{j=1,\ldots,D,\, |\al|\leq m}$ (see \S\ref{subsec:main}). Let $d := D\binom{m+n}{m}$.

Trivially, $F$ is a solution of a system of equations
\begin{equation}\label{eq:Y}
A_1(x)\cdot F_x = B_1(x),\quad x\in X,
\end{equation}
where $A_1(x)$ is a $d \times d$ matrix of semialgebraic 
functions on $X$, and $B_1(x)$ is an $\IR^d$-valued semialgebraic function on $X$, defined as follows:
\begin{equation*}
A_1(x) := \begin{cases} I_d, &x\in Y,\\
                                       0,   &x\in X\backslash Y,
                \end{cases}   
\qquad\quad B_1(x) :=  \begin{cases} F_x, &x\in Y,\\
                                       0,   &x\in X\backslash Y,
                \end{cases}                                                    
\end{equation*}
where $I_d$ is the $d\times d$ identity matrix. The system of equations \eqref{eq:Y} is clearly multiplicative.

Therefore, 
$F$ is a solution of a system of equations
\begin{equation}\label{eq:Ybig}
\tA(x)\cdot F_x = \tB(x),\quad x\in X,
\end{equation}
where
\begin{equation*}
\tA(x) := \left(\begin{matrix} A(x)\\
                                           A_1(x)
                    \end{matrix}\right),
\qquad\quad \tB(x) :=  \left(\begin{matrix} B(x)\\
                                           B_1(x)
                    \end{matrix}\right).
\end{equation*}                    
Since \eqref{eq:eq} and \eqref{eq:Y} are both multiplicative, so is \eqref{eq:Ybig}.                     
                                          
By Theorem \ref{thm:main}, \eqref{eq:Ybig} has a semialgebraic $\cC^m$ Whitney field solution $G$ on $X$;
in particular, $G$ is a solution of \eqref{eq:eq} and $G_x = F_x$, $x\in Y$.
\end{proof}

\begin{remark}\label{rem:mainM} 
By essentially the same argument, Theorem \ref{thm:mainjet} implies a relative version as a corollary.

To obtain a relative version of Corollary \ref{cor:mainM}, consider a solution $f \in \cC^m(\IR^n, \IR^D)$ of a
system of equations \eqref{eq:mainM} and assume that $f$ is semialgebraic on $Y\subset X$. Then we can
use the argument of the proof of Corollary \ref{cor:relmain} with $A_1(x) := I_D$, $B_1(x)=f(x)$, if $x\in Y$,
and $A_1(x) := 0$, $B_1(x)=0$, if $x\in X\backslash Y$, to obtain a $\cC^m$ semialgebraic solution $g$
of  \eqref{eq:mainM}, with  $g|_Y = f|_Y$.
\end{remark}

\begin{proof}[Proof of Corollary \ref{cor:mainW2}]
As in Remarks \ref{rem:closure}, we can assume that $X\subset \IR^2$ is closed. Let $U$ denote the interior of $X$,
$Z:= X\backslash U$ and $Y:= \bdry U$. Then $Z$ is a closed semialgebraic set of dimension one, and $Y$ is a closed
semialgebraic subset of $Z$.

Let $F$ denote the Taylor field of $f$ on $X$. Then $F$ is a $\cC^m$ Whitney field, and $F|_{\overline{U}}$, 
$H := F|_Y$ are semialgebraic. Let $d := \binom{m+2}{m}$.

Now, $F|_Z$ is a solution of a system of equations on $Z$,
\begin{equation}\label{eq:W2}
A(x)\cdot F_x = B(x),\quad x\in Z,
\end{equation}
where $B(x)$ is an $\IR^d$-valued semialgebraic function on $Z$, and $A(x)$ is a $d \times d$ matrix of semialgebraic 
functions on $Z$, say $B(x) = (b_\al(x))_{\al\in\IN^2,\,|\al|\leq m}$, and
$A(x) := (a_{\al\be}(x))_{|\al|,|\be|\leq m}$,
defined as follows:
\smallskip
\begin{equation*}
\begin{aligned}
&B(x) = H_x, &&\qquad \text{if $x\in Y$,}\\
&b_0(x) = g(x) \text{ and $b_\al(x)=0$, $\al\neq 0$}, &&\qquad \text{if $x\in Z\backslash Y$,}
\end{aligned}
\end{equation*}
and
\begin{equation*}
\begin{aligned}
&A(x) := I_d,   
&&\qquad \text{if $x\in Y$,}\\
&a_{00}(x) = 1 \text{ and $a_{\al\be}(x) = 0$, $(\al,\be) \neq (0,0)$}, &&\qquad \text{if $x\in Z\backslash Y$.}
\end{aligned}
\end{equation*}

\medskip
The system of equations \eqref{eq:W2} is multiplicative since $\Ker A(x) = 0$, if $x\in Y$, and
$\Ker A(x) = \{P\in\cP: P(0)=0\}$, if $x\in Z\backslash Y$.

By Corollary \ref{cor:relmain}, \eqref{eq:W2} has a semialgebraic $\cC^m$ Whitney field solution $G$ on $Z$,
such that $G = H$ on $Y$. Then there is a semialgebraic $\cC^m$ Whitney field $\tG$ on $X$, such that
$\tG = G$ on $Z$, and $\tG = F$ on $\overline{U}$ ($Z$ and $\overline{U}$ are trivially $1$-regularly situated). 
By \cite{KP}, there is a semialgebraic function $\tg\in \cC^m(\IR^2)$
with Taylor field $\tG$ on $X$; $\tg$ is an extension of the given semialgebraic function $g:X\to \IR$ in 
Corollary \ref{cor:mainW2}.
\end{proof}

Alternatively, Corollary \ref{cor:mainW2} can be deduced in essentially the same way from the relative version of Theorem \ref{thm:mainjet}.

\section{The bundle viewpoint}\label{sec:bundle}
Recall the identification $\cP^D\cong \IR^d$ of \S\ref{subsec:notn}. We give $\cP^D$ the Euclidean structure induced by
that of $\IR^d$.

\begin{proposition}\label{prop:bundle}
Let $\cH=(H(x))_{x\in X}$ be a semialgebraic $\cC^m(\IR^n,\IR^D)$-bundle on a subset $X\subset\IR^n$.
Then there exist a $d\times d$ matrix $A(x)$ of semialgebraic functions on $X$, and a vector-valued semialgebraic function
$B:X\to\IR^d$, such that
\begin{equation*}
H(x)=\left\{P\in\cP^D: A(x)\cdot P=B(x)\right\},\quad  x\in X;
\end{equation*}
i.e., $\cH$ is the bundle of solutions of a system of equations \eqref{eq:eq} with $k=d$.
Moreover, this system of equations is multiplicative, and admits a solution $P$ for every $x\in X$.
\end{proposition}

\begin{proof}
By definition, given $x\in X$, the fibre $H(x)$ is a coset of a $\cP$-submodule $I(x)$ of $\cP^D$; thus 
\begin{equation*}
I(x) = H(x)-H(x) := \left\{P-Q: P,Q\in H(x)\right\},
\end{equation*}
and $H(x)=P+I(x)$, for any $P \in H(x)$.

Let $\Pi(x)$ denote the matrix of the orthogonal projection of $\cP^D$ onto $I(x)$. The graph $\Ga_\Pi$ of $\Pi$ is a
subset of $X\times M_d(\IR)$, where $M_d(\IR)$ denotes the space of $d\times d$ real matrices; $\Ga_\Pi$ can be
defined by the following first-order formula involving only semialgebraic objects:
\begin{equation*}
\left\{(x,M)\in X\times M_d(\IR): \forall P\in\cP^D,\,MP\in I(x)\text{ and }\forall Q\in I(x),\,(P-MP)\cdot Q=0\right\}.
\end{equation*}
Therefore, $\Pi$ is semialgebraic, and so is the function $A(x) := I_d-\Pi(x)$.

By semialgebraic selection (or definable choice), there is a semialgebraic mapping $\sigma:X\to\cP^D$ such that $\sigma(x)\in H(x)$
for all $x\in X$. Set $B(x) := A(x)\cdot\sigma(x)$. Then $B$ is semialgebraic. Since $H(x)=\sigma(x)+I(x)=\sigma(x)+\Ker A(x)$, we get
\begin{equation*}
H(x)=\left\{P\in\cP^D: A(x)\cdot P=B(x)\right\}.
\end{equation*}
The system of equations constructed is multiplicative, since $\Ker A(x)=I(x)$ is a $\cP$-submodule of $\cP^D$,
and admits a solution $\s(x)$, for every $x\in X$.
\end{proof}

\section{O-minimal generalization}\label{sec:omin}
The  purpose of this final section is to indicate how to generalize our results from semialgebraic functions
to functions that are definable in an o-minimal structure on $\IR$. We show, in particular, how to replace the
use of Puiseux's theorem (i.e., the use of \eqref{eq:puis}) in the proof of the main theorem. 

Fix an o-minimal expansion of $\IR$. Suppose $f: A\to \IR$ is definable on a subset $A$ of $\IR$. Then there is a finite
subset $S$ of $A$ such that $A\backslash S$ is a finite union of open intervals, on each of which $f$ is (strictly)
increasing, decreasing or constant (by the \emph{monotonicity theorem}, \cite[Chapt.\,3]{vdD}).

A definable set need not have an analytic or even $\cC^\infty$ cell decomposition, but, for every positive integer
$k$, it has a $\cC^k$ cell decomposition. ($\cC^{2m}$ suffices for Lemma \ref{lem:FL5.1} and Theorem \ref{thm:main}.)
In particular, for given $k$, $f$ is $\cC^k$ outside a finite subset of $A$,
and the $k$th derivative $f^{(k)}$ is definable \cite[Chapt.\,7]{vdD}.

\begin{lemma}\label{lem:def}
Let  $g,\,u:(0,\varepsilon)\to\IR$ be definable $\cC^k$ functions.
\begin{enumerate}
\item If $\lim_{x\to0+}g(x)\in\IR$ then $\lim_{x\to0+}xg'(x)=0$.

\smallskip
\item If $\lambda\in\IZ$ and $g(x)=\o{x\to0+}(x^\lambda)$ then $g^{(j)}(x)=\o{x\to0+}(x^{\lambda-j})$, $j=1,\ldots,k$.
\smallskip
\item Assume that $u>0$ and $\lim_{x\to 0+}u(x)=0$. If $\lambda\in\IZ$ and 
$g(x)=\o{x\to0+}(u(x)^\lambda)$, then $g'(x)=\o{x\to0+}(u(x)^{\lambda-1}u'(x))$. 
If, moreover, $u(x)\leq Cx$ on $(0,\varepsilon)$, where $C>0$, then $g^{(j)}(x)=\o{x\to0+}(u(x)^{\lambda-j})$, $j=1,\ldots,k$.
\end{enumerate}
The statements (2) and (3) also hold with $\o{}$ replaced by $\O{}$.
\end{lemma}

\begin{proof}
(1) Since $x\mapsto xg'(x)$ is definable, $L := \lim_{x\to 0+}xg'(x)\in\IR\cup\{\pm\infty\}$ is well defined.
Suppose $L\neq0$; say $L>0$. Then $g'(x)\geq c/x$ ($c>0$) for $x$ small enough: thus $g(x)\leq c\log x + \text{const}$,
in contradiction to $\lim_{x\to0+}g(x)\in\IR$.

\medskip\noindent
For (2) and (3), by induction, it is enough to consider the case $j=1$. 

\medskip\noindent
(2) This follows from (1) applied to $g(x)/x^\lambda$.

\medskip\noindent
(3) We can assume that $u$ is increasing.
Therefore, there exists $\delta\in(0,\varepsilon)$ such that $u:(0,\delta)\to(0,u(\delta))$ is a definable $\cC^1$-diffeomorphism.
Then $G := g\circ u^{-1}$ is a definable $\cC^1$ function on $(0,u(\delta))$, and 
\begin{equation*}
G(t)=g\left(u^{-1}(t)\right)=\o{t\to0+}\left(t^\lambda\right).
\end{equation*}
By (2), $G'(t)=\o{t\to0+}(t^{\lambda-1})$, so that
\begin{equation*}
g'(x)=G'\left(u(x)\right)u'(x)=\o{x\to0+}\left(u(x)^{\lambda-1}u'(x)\right). 
\end{equation*}
If, moreover, $u\le Cx$ then $\lim_{x\to 0+}u'(x)=\lim_{x\to0+} u(x)/x \leq C$, by l'H\^opital's rule (which holds
by the monotonicity theorem), and the result follows.
\end{proof}

\begin{remark}\label{rem:def}
Item (2) of the lemma (and its $\O{}$ version) replaces the use of \eqref{eq:puis} in the proof of Claim \ref{claim:first};
it gives
\begin{equation*}
\p_x^q f^s_{j,\be}(x)=\o{x\to0+}\left(x^{m-|\be|-q}\right),\qquad \p_x^p\psi_s(x)=\O{x\to0+}\left(x^{1-p}\right).
\end{equation*}

The use of \eqref{eq:puis} in the proof of \eqref{eq:(1)} can be replaced by the second assertion of item (3), applied to $u=\rho_{st}$;
this gives
\begin{equation*}
W'_{\ell,\alpha}(x)=\o{x\to0+}\left(\rho_{st}(x)^{m-|\alpha|-\ell-1}\right).
\end{equation*}
We note that the first equality in \eqref{eq:(1)}, which follows from \eqref{eq:puis} in the semialgebraic case (or more generally in any polynomially bounded o-minimal expansion of $\IR$), is nevertheless not true in $\IR_{\exp}$. Indeed, consider
$\psi_s := 0$ and $\psi_t(x) := e^{-1/x}$, so that $\rho_{st}(x)=e^{-1/x}$. If $g(x):= \sqrt{x}\,\rho_{st}(x)^N$,
where $N\geq 1$, then $g(x) =\o{x\to0+}\left(\rho_{st}(x)^N\right)$, and $g'(x)=\o{x\to0+}\left(\rho_{st}(x)^{N-1}\right)$,
by Lemma \ref{lem:def}\,(3), but
\begin{equation*}
g'(x)=\left(\frac{1}{2\sqrt{x}}+ \frac{N}{x^{3/2}}\right)e^{-N/x}
\end{equation*}
is not $\o{x\to0+}\left(\rho_{st}(x)^Nx^{-1}\right)$.
\end{remark}

\bibliographystyle{amsplain}

\end{document}